\documentclass{amsart}
\usepackage{amsmath,amssymb}
\usepackage{cite}
\usepackage{array}
\usepackage{tabularx}
\usepackage{tabulary}
\usepackage{stmaryrd}
\usepackage{ragged2e}
\usepackage{hyperref}
\usepackage[abbrev,msc-links]{amsrefs}
\usepackage[all,cmtip]{xy}

\newcommand{\Ext}{\operatorname{Ext}}
\newcommand{\Tor}{\operatorname{Tor}}
\newcommand{\Hom}{\operatorname{Hom}}

\newcommand{\HH}{{\operatorname H}}
\newcommand{\dell}{\partial}

\newcommand{\coker}{\operatorname{coker}}

\newcommand{\fm}{{\mathfrak{m}}}

\newcommand{\rank}{\operatorname{rank}}

\newcommand{\pd}{\operatorname{pd}}

\newcommand{\ld}{\operatorname{ld}}

\newcommand{\ZZ}{\mathbb{Z}}
\newcommand{\NN}{\mathbb{N}}
\newcommand{\RR}{\mathbb{R}}
\newcommand{\QQ}{\mathbb{Q}}

\newcommand{\FF}{\mathbf{F}}
\newcommand{\GG}{\mathbf{G}}

\newcommand{\CC}{\mathbf{C}}

\newcommand{\xra}{\xrightarrow}
\newcommand{\onto}{\twoheadrightarrow}

\newcommand{\gr}{\operatorname{gr}}
\newcommand{\grR}{\operatorname{gr}_\mathfrak{m}(R)}
\newcommand{\grM}{\operatorname{gr}_\mathfrak{m}(M)}
\newcommand{\ds}{\displaystyle}

\theoremstyle{plain}
\newtheorem{theorem}{Theorem}[section]

\newtheorem{proposition}[theorem]{Proposition}
\newtheorem{lemma}[theorem]{Lemma}
\newtheorem{corollary}[theorem]{Corollary}
\newtheorem{fact}[theorem]{Fact}

\newtheorem*{TheoremA}{Theorem A}
\newtheorem*{TheoremB}{Theorem B}

\theoremstyle{definition}
\newtheorem{definition}[theorem]{Definition}
\newtheorem{definitions}[theorem]{Definitions}
\newtheorem{example}[theorem]{Example}

\newtheorem{setup}[theorem]{Setup}

\theoremstyle{remark}
\newtheorem{remark}[theorem]{Remark}
\newtheorem{remarks}[theorem]{Remarks}
\newtheorem{question}[theorem]{Question}

\newtheoremstyle{myplain}
     {1em plus 0.15em minus 0.05em}
     {1em plus 0.15em minus 0.05em}
     {\itshape}
     {}
     {\bf}
     {.}
     {0.5em}
     {}

\newtheoremstyle{mydef}
     {1em plus 0.15em minus 0.05em}
     {1em plus 0.15em minus 0.05em}
     {}
     {}
     {\bf}
     {.}
     {0.5em}
     {}

\newtheoremstyle{myrmk}
     {1em plus 0.15em minus 0.05em}
     {1em plus 0.15em minus 0.05em}
     {}
     {}
     {\itshape}
     {.}
     {0.5em}
     {}

\theoremstyle{myplain}

\theoremstyle{myrmk}

\theoremstyle{mydef}

\begin{document}
\title[Eventually Linear Partially Complete Resolutions]{Eventually Linear Partially Complete Resolutions over a Local Ring with $\fm^4=0$}
\author[K.\ A.\ Beck]{Kristen A.\ Beck}
\address{Kristen A.\ Beck,
Department of Mathematics, The University of Arizona,
617 N.\ Santa Rita Ave.,
Tucson, AZ 85721 U.\ S.\ A.}
\email{kbeck@math.arizona.edu}

\thanks{{\em Date:} \today.}

\maketitle

\begin{abstract}
 {We} classify the Hilbert polynomial of a local ring $(R,\fm)$ satisfying $\fm^4=0$ which admits an eventually linear resolution $\CC$ which is `partially' complete --- that is, for which $\HH^i\Hom_R(\CC,R)$ vanishes for all $i\gg 0$.  As a corollary to our main result,  {we} show that an $\fm^4=0$ local ring can admit certain classes of asymmetric partially complete resolutions only if its Hilbert polynomial is symmetric.  Moreover,  {we} show that the Betti sequence associated to an eventually linear partially complete resolution over an $\fm^4=0$ local ring cannot be periodic of period two or three.
\end{abstract}

\section*{Introduction}

Let $(R,\fm)$ be a commutative local Noetherian ring with unique maximal ideal $\fm$.  A finitely generated $R$-module $M$ is said to be totally reflexive if it is reflexive and if both $\Ext_R^i(M,R)$ and $\Ext_R^i(\Hom_R(M,R),R)$ vanish for all $i>0$. It is straightforward to see that every free $R$-module is totally reflexive, but not \emph{vice versa}.  For example, over a Gorenstein local ring, the maximal Cohen-Macaulay modules and totally reflexive modules coincide.  Furthermore, while there are several known examples of non-free totally reflexive modules over non-Gorenstein rings (cf.\ \cites{AvGaPe97,Be12,TakWat07,Vel02}, among others), the classification of all local rings which admit such modules is far from complete.

Perhaps the structurally simplest non-Gorenstein local rings $(R,\fm)$ to admit non-free totally reflexive modules are those satisfying the condition $\fm^3=0\,(\neq\fm^2)$.  In 2003, Yuji Yoshino provided a characterization of such rings; cf.\ \cite[Theorem 3.1]{Yo03}.  In particular, the author proved that if a local ring $(R,\fm)$ satisfying $\fm^3=0$ admits a non-free totally reflexive module $M$, then (a) the Betti numbers of $M$ are constant, (b) $R$ is Koszul, and (c) the Hilbert polynomial of $R$ is balanced --- that is, it has a root of $-1$.


One can see that Yoshino's characterization in \cite{Yo03} does not extend to the class of local rings $(R,\fm)$ satisfying $\fm^4=0\,(\neq\fm^3)$; indeed, there are many examples of totally reflexive modules with non-constant Betti numbers over such rings.
Even more noteworthy, however, is an example constructed by Jorgensen and \c{S}ega in 2005, in which the authors exhibit a local ring $(R,\fm)$ satisfying $\fm^4=0$ and a totally reflexive $R$-module $M$ such that the Betti numbers of $M$ are constant, while the Betti numbers of $\Hom_R(M,R)$ grow exponentially; \cite{JoSeg05}*{Theorem 1.2}.

The results in this paper are motivated by those of both \cite{Yo03} and \cite{JoSeg05}, but with respect to a more general class of modules: those with eventually linear partially complete resolutions, where the latter characteristic is defined by the vanishing of $\Ext_R^i(M,R)$ and $\Ext_R^i(\Hom_R(M,R),R)$ for all $i\gg 0$.  Given a local ring $(R,\fm)$ satisfying $\fm^4=0$ which admits such a module $M$,  {we} would like to (a) characterize the Hilbert polynomial of $R$, and (b) give necessary conditions for $M$ to possess an asymmetric (partially) complete resolution.  However, many of our results concerning the Hilbert polynomial of $R$ regard a still more general class of modules: those which satisfy at least the vanishing of $\Ext_R^i(M,R)$ for $i\gg 0$.

In Sections \ref{sec:prelim} and \ref{sec:syseq}  {we} outline preliminary definitions and concepts which are extensively used throughout the paper.  Section \ref{sec:betti} is concerned with the growth of the Betti sequence; in addition to identifying particular growth rates, this section builds up to the following result.

\begin{TheoremA}
Let $M$ be a finitely generated module with an eventually linear minimal free resolution over a local ring $(R,\fm)$ satisfying $\fm^4=0$.  If $\Ext_R^*(M,R)$ eventually vanishes, then the Betti sequence of $M$ is not eventually periodic of period two or three.
\end{TheoremA}

The main goal of Section \ref{sec:hilbert} is to characterize the Hilbert polynomial of an $\fm^4=0$ local ring admitting eventually linear partially complete resolutions with certain types of acyclicity. 
However, as mentioned above, most of our results --- including the following --- are stated in a more general setting.

\begin{TheoremB}\label{thm:mainpolyexp}
Let $M$ be a finitely generated module with an eventually linear minimal free resolution over a local ring $(R,\fm)$ satisfying $\fm^4=0$, and suppose that $\Ext_R^*(M,R)$ eventually vanishes.
\begin{enumerate}
\item If the Betti sequence of $M$ has non-exceptional polynomial growth, then the Hilbert polynomial of $R$ is symmetric.
\item If the Betti sequence of $M$ has exponential growth of base $a$, then the Hilbert polynomial of $R$ takes the form $H_R(t)=1+et+ft^2+gt^3$, where
\[
f=\left(a+\frac{1}{a}\right)e-\left(a^2+1+\frac{1}{a^2}\right)\quad\text{and}\quad g=e-\left(a+\frac{1}{a}\right).
\]
\end{enumerate}
\end{TheoremB}

In particular, we use the above result to prove, in Theorem \ref{thm:polyexp}, that the type of acyclic complete resolution exhibited in \cite{JoSeg05} can only be admitted by a ring with a symmetric Hilbert polynomial.  We further prove, in Theorem \ref{thm:expexp}, that certain exponential vs.\ exponential acyclic complete resolutions do not exist over $\fm^4=0$ local rings.

\section{Preliminary Concepts}\label{sec:prelim}

\numberwithin{theorem}{subsection}
\numberwithin{equation}{theorem}

Unless otherwise stated, $R$ shall represent the commutative local (Noetherian) ring $(R,\fm,k)$ having unique maximal ideal $\fm$ and residue class field $k:=R/\fm$.  Furthermore, whenever a local ring satisfies the condition $\fm^4=0$, it shall also be assumed that $\fm^3\neq 0$.

\subsection{Total reflexivity}

Although the concept of total reflexivity can be defined more generally over an arbitrary Noetherian ring, cf. \cite{AvMar02}*{Section 2}, our results are concerned with its existence over a certain class of local rings.

\begin{definition}\label{def:totreflexive}
A finitely generated $R$-module $M$ is said to be \emph{totally reflexive} if each of the following conditions hold.
\begin{enumerate}
\item The canonical map $M\to\Hom_R(\Hom_R(M,R),R)$ is an isomorphism.
\item $\Ext_R^i(M,R)=0$ for all $i>0$.
\item $\Ext_R^i(\Hom_R(M,R),R)=0$ for all $i>0$.
\end{enumerate}
\end{definition}

 {We} denote by $M^*$ the ring dual $\Hom_R(M,R)$.  Notice that condition (1) above yields the classical reflexivity condition $M\cong M^{**}$.

\begin{remark}\label{rmk:conditions}
There exist finitely generated reflexive modules which are not totally reflexive; for example, if $R=k\llbracket x_1,x_2,\ldots,x_n\rrbracket$, then $\Omega_2(k)$ fits the bill.  Still, the overall independence of conditions (1)--(3) in Definition \ref{def:totreflexive} is not completely understood over an arbitrary local ring.  In \cite{JoSeg06}, however, Jorgensen and \c{S}ega demonstrate the independence of conditions (2) and (3) for a finitely generated reflexive $R$-module $M$.
\end{remark}

\begin{definition}
Given a totally reflexive $R$-module $M$, let $\FF=(F_i,\dell_i)$ and $\GG=(G_i,d_i)$ be minimal free resolutions of $M$ and $M^*$, respectively.  Then the totally acyclic complex defined by
\[
\xymatrixrowsep{10pt}\xymatrixcolsep{7pt}\xymatrix{
\cdots\ar@{->}[rr]
&&F_2\ar@{->}[rr]^{\dell_2}
&&F_1\ar@{->}[rr]^{\dell_1}
&&F_0\ar@{->}[rr]^{\delta}\ar@{->>}[dr]_{\pi}
&&G_0^*\ar@{<-_{)}}[dl]^{\iota}\ar@{->}[rr]^{d_1^*}
&&G_1^*\ar@{->}[rr]^{d_2^*}
&&G_2^*\ar@{->}[rr]
&&\cdots\\
&&&&&&&M&&&&&&&&
}
\]
such that $\delta:=\iota\circ\pi$, is called the \emph{complete resolution} of $M$ over $R$, and is often denoted $\FF|\GG^*$.
\end{definition}

In fact, given any totally acyclic complex $\CC=(C_i,\dell_i)_{i\in\ZZ}$ over $R$, the $R$-module $\coker\dell_j$ is totally reflexive for each $j\in\ZZ$. 

\subsection{Linear resolutions}


%

Let $\CC=(C_i,\dell_i)_{i\in\ZZ}$ be a minimal complex of $R$-modules. Then the \emph{associated graded complex} of $\CC$ with respect to $\fm$ is given by
\[
\gr_{\fm}(\CC):=\bigoplus_{j\in\ZZ}\,\gr_{\fm}(\CC)^j
\]
where
\[
\gr_{\fm}(\CC)^j:\quad \cdots\longrightarrow\frac{\fm^{j-i-1}C_{i+1}}{\fm^{j-i}C_{i+1}}\xra{\,\,\delta^j_{i+1}\,\,}\frac{\fm^{j-i}C_i}{\fm^{j-i+1}C_i}\xra{\,\,\delta^j_i\,\,}\frac{\fm^{j-i+1}C_{i-1}}{\fm^{j-i+2}C_{i-1}}\longrightarrow\cdots
\]
such that, for each $ d\in\ZZ$, the map $\delta^j_ d$ is induced by the restriction of $\fm^{j- d}C_ d\to\fm^{j- d+1}C_{ d-1}$, modulo $\fm^{j- d+1}C_ d$. {We} also maintain the convention that $\fm^n=R$ for $n\leq 0$.

\begin{remark}
The associated graded complex is often referred to as being the \emph{linear part} of a minimal complex, precisely because it filters non-linear behavior from the differentials.  As one would imagine, this sort of a construction does not always preserve the exactness of a complex.
\end{remark}

\begin{definitions}
Let $M$ be a finitely generated $R$-module with minimal free resolution $\FF\onto M$.  If $\gr_{\fm}(\FF)$ is exact in positive degrees, then $\FF$ is said to be a \emph{linear resolution} of $M$.  In addition, if $M$ is totally reflexive and the minimal free resolution $\GG\onto M^*$ is linear, then $\FF|\GG^*$ is said to be a \emph{linear complete resolution} of $M$.

In general, given a minimal free resolution $\FF\onto M$ over $R$, the quantity defined by
\[
\ld_R(M):=\sup\left\{n\mid\HH_i\left(\gr_{\fm}(\FF)\right)= 0\text{ for all }i\geq n\right\}
\]
is called the \emph{linearity defect} of $M$.  If $\ld_R(M)$ is finite and positive, then $\FF$ is said to be \emph{eventually linear}. If, given a totally reflexive module $M$, the quantities $\ld_R(M)$ and $\ld_R(M^*)$ are both finite and positive, then the complete resolution of $M$ is called eventually linear.
\end{definitions}

\begin{remark}
Simply stated, a minimal complex of free $R$-modules is (eventually) linear if and only if the nonzero entries of the matrices defining each (sufficiently high) differential in the complex are elements of $\fm\setminus\fm^2$.
\end{remark}

\begin{definition}
If the minimal free resolution $\FF\onto M$ is linear, then $M$ is said to be a \emph{Koszul module}.  Notice that if, instead, $\ld_R(M)=n<\infty$, then the $n$th syzygy module of $M$, denoted $\Omega_n(M)$, is Koszul.  Furthermore, $(R,\fm,k)$ is called a \emph{Koszul ring} whenever $k$ is Koszul as an $R$-module.
\end{definition}

\subsection{Hilbert series and Poincar\'{e} series}

Let $k$ be an arbitrary field, and suppose that $V=\bigoplus_{i\geq 0}V_i$ is a graded vector space over $k$ such that $\dim_kV_n<\infty$ for each $n\in\NN$.  Then recall that the formal power series given by
\[
H_V(t):=\sum_{i\geq 0}\dim_{k}\!V_i\,t^i\in\ZZ\llbracket t\rrbracket
\]
is called the \emph{Hilbert series} of $V$.  Thus, the Hilbert series of the associated graded ring of $(R,\fm,k)$ is given by
\[
H_{\grR}(t):=\sum_{i\geq 0}\dim_k\left(\fm^i/\fm^{i+1}\right)t^i.
\]
One should note that if $\fm^n=0$ for some $n<\infty$, then the Hilbert series of $\grR$ is simply a polynomial.

Furthermore, for any finitely generated $R$-module $M$, one can also speak of the Hilbert series of the associated graded module of $M$, which is defined by
\[
H_{\grM}(t):=\sum_{i\geq 0}\dim_k\left(\fm^i M/\fm^{i+1}M\right)t^i.
\]



If $\FF\onto M$ is a minimal free resolution, then the quantity
\[
\rank F_n:=\dim_k\Tor_n^R(M,k)
\]
is called the $n$th \emph{Betti number} of $M$, and shall be denoted $b_n(M)$, or simply $b_n$ when there is no risk of confusion as to the module.  Often throughout this paper,  {we} will refer to the \emph{Betti sequence} of $M$ over $R$, which is given by $\{b_i(M)\}_{i\geq 0}$.

Furthermore, the formal power series
\[
P_M^R(t):=\sum_{i\geq 0}b_i(M)\,t^i\in\ZZ\llbracket t\rrbracket
\]
is called the \emph{Poincar\'{e} series} of $M$ over $R$.

\begin{fact}
If $(R,\fm,k)$ is Koszul, then the Poincar\'{e} series of $k$ as an $R$-module is given by
\[
P_k^R(t)=\frac{1}{H_{\grR}(-t)}
\]
\end{fact}

In the next section,  {we} develop a system of equations which will be used to prove each of our main results.

\section{The System of Equations}\label{sec:syseq}

\numberwithin{theorem}{section}
\numberwithin{equation}{theorem}

The following lemma is a result of Herzog and Iyengar which characterizes, in terms of the ring alone, the Poincar\'{e} series of a finitely generated module with an eventually linear minimal free resolution.   {We} state it without proof.

\begin{lemma}\label{lemma:linres}\cite{HeIy05}*{Proposition 1.8}
Let $M$ be a finitely generated module over a local ring $(R,\fm)$.  If the minimal free resolution of $M$ is eventually linear, then
\[
P_M^R(t)=\frac{q(t)}{H_{\grR}(-t)(1+t)^{\dim R}}
\]
for some $q\in\ZZ[t]$.
\end{lemma}

Note that if $\Ext_R^i(M,R)$ vanishes for all $i\gg 0$, the above lemma will characterize the Poincar\'{e} series of a sufficiently high syzygy in the variable $\frac{1}{t}$.  Furthermore assuming that $R$ is zero-dimensional, one obtains the following result.

\begin{proposition}\label{prop:linres}
Let $M$ be a finitely generated module over a zero-dimensional local ring $(R,\fm)$ such that $\Ext_R^i(M,R)=0$ for all $i\gg 0$.  If the minimal free resolution of $M$ is eventually linear, then there exists $n\in\NN$ such that the following hold.
\begin{enumerate}
\item $P_{\Omega_n(M)}^R(t)\,H_{\grR}(-t)\in\ZZ[t]$\\[-0.1in]
\item $P_{\Omega_n(M)}^R(\textstyle{\frac{1}{t}})\,H_{\grR}(-t)\in\ZZ[t]$
\end{enumerate}
\end{proposition}

\begin{proof}
(1) follows from Lemma \ref{lemma:linres}, a proof of which can be found in \cite{HeIy05}.

In order to prove (2), let $n\in\NN$ be such that $\Ext_R^i(M,R)=0$ for $i>n$.  Then, letting $(R^{b_i},\dell_i)$ denote the minimal free resolution of $M$, one has that the sequence
\[
0\to\Omega_n(M)^*\to R^{b_n}\xra{\dell_{n+1}^*}R^{b_{n+1}}\xra{\dell_{n+2}^*}R^{b_{n+2}}\xra{\dell_{n+3}^*}R^{b_{n+3}}\to\cdots
\]
is exact.  The additivity of Hilbert series on short exact sequences now yields
\[
H_{\gr_{\fm}(\Omega_n(M))}(t)=H_{\grR}(t)P_{\Omega_n(M)}^R(-\textstyle{\frac{1}{t}}).
\]
The rest of the proof follows from that of (1).
\end{proof}

\begin{setup}\label{setup:system}
 {We} shall henceforth restrict our attention to the class of local rings $(R,\fm)$ satisfying $\fm^4=0$.  And, for the sake of simplicity,  {we} will abuse notation in the sequel by allowing
\[
H_R(t)=1+et+ft^2+gt^3
\]
to represent the Hilbert polynomial of $\grR$.  Furthermore,  {we} will often refer to this polynomial as being the Hilbert polynomial of $R$ itself.

Now suppose that $M$ is a finitely generated $R$-module which admits an eventually linear minimal free resolution and has Betti sequence $\{b_i\}_{i\geq 0}$.  If $\Ext_R^i(M,R)$ vanishes for all $i\gg 0$, then Proposition \ref{prop:linres} yields the following system of equations for each $i,j\gg 0$.
\begin{equation}\label{setup:matrixeq}
\left[\begin{array}{ccc}b_i&-b_{i+1}&b_{i+2}\\b_{j+3}&-b_{j+2}&b_{j+1}\end{array}\right]\left[\begin{array}{c}g\\f\\e\end{array}\right]=\left[\begin{array}{c}b_{i+3}\\b_j\end{array}\right]
\end{equation}
Row reduction on (\ref{setup:matrixeq}) yields the reduced system
\begin{equation}\label{setup:redmatrixeq}
\left[\begin{array}{ccc}\vspace{0.02in}b_i & -b_{i+1} & b_{i+2}\\0 & \displaystyle{\frac{\Delta_1(i,j)}{b_i}} & \displaystyle{\frac{\Delta_2(i,j)}{b_i}}\end{array}\right]
\left[\begin{array}{c}g\\f\\e\end{array}\right]=\left[\begin{array}{c}\vspace{0.02in}b_{i+3}\\\displaystyle{\frac{\Delta_3(i,j)}{b_i}}\end{array}\right]
\end{equation}
where
\begin{align}\label{setup:deltas}
\Delta_1(i,j)&=b_{i+1}b_{j+3}-b_i b_{j+2}\notag\\
\Delta_2(i,j)&=b_i b_{j+1}-b_{i+2}b_{j+3}\\
\Delta_3(i,j)&=b_i b_j-b_{i+3}b_{j+3}\notag
\end{align}
for each $i,j\gg 0$.
\end{setup}

Given the Betti sequence of such an $R$-module, one can use Setup \ref{setup:system} to determine the explicit Hilbert polynomial of $R$.  However, since \eqref{setup:matrixeq} represents infinitely many systems of equations, it might not be immediately obvious that this method is at all well-defined.  Nevertheless, Proposition \ref{prop:linres} insists that, given the Betti sequence $\{b_i\}_{i\geq 0}$ of some $R$-module, a solution to \eqref{setup:matrixeq} will be unique.

The Hilbert polynomial therefore hinges on what possible forms the Betti sequence can take on.  This question is discussed in the following section.

\section{The Betti Sequence}\label{sec:betti}

\numberwithin{theorem}{subsection}
\numberwithin{equation}{theorem}

There is very little known about the asymptotic behavior of the Betti sequence of a finitely generated module over an arbitrary local ring.  In particular, an answer to the following question of Avramov is still unknown in general.

\begin{question}\cite{Av10}*{4.3.3}
Is the Betti sequence of a finitely generated module over a local ring eventually non-decreasing?
\end{question}

\noindent A negative answer to this question would introduce the possibility of asymptotic periodicity in the Betti sequence.   {We} consider this question next.

\subsection{Periodicity}

Our main result for this section will show that Betti sequences associated with linear resolutions over an $\fm^4=0$ local ring cannot have `small' periodicity.  {We} make this statement precise in Theorem \ref{thm:periodicity} below; first, however,  {we} present the following fact.

\begin{fact}\label{fact:vanishingdeltas}
Let $M$ be a finitely generated $R$-module with Betti sequence $\{b_i\}_{i\geq 0}$.  Furthermore, for each $ d\in\{1,2,3\}$, let $\Delta_ d(i,j)$ be the quantity defined in (\ref{setup:deltas}).  Then $\Delta_ d(i,j)=0$ for all $i,j\gg 0$ if and only if there exists a positive integer $n$, which divides $d$, such that $\{b_i\}_{i\geq 0}$ is eventually periodic of period $n$.
\end{fact}

\begin{proof}
Fix $ d\in\{1,2,3\}$ and suppose that $b_i b_{j- d+3}=b_{i+ d}b_{j+3}$ for all $i,j\gg 0$.  Choosing $j=i+ d-3$, one has that $\{b_i\}_{i\geq 0}$ eventually satisfies $b_i^2=b_{i+ d}^2$, which implies that $b_i=b_{i+ d}$ for all $i\gg 0$.  This implies that $\{b_i\}_{i\ge 0}$ is periodic, and its period clearly divides $d$.

Conversely, fix $ d\in\{1,2,3\}$ and suppose that there exists $n\in\ZZ^+$, with $n\mid d$, such that $\{b_i\}_{i\geq 0}$ is periodic with period $n$.  Then in particular $b_i=b_{i+ d}$, and moreover $b_ib_{j- d+3}=b_{i+ d}b_{j+3}$, for all $i,j\gg 0$.  The result follows.\qedhere
\end{proof}

\begin{remarks}\label{rmk:delta}
(1) The vanishing of $\Delta_ d(i,j)$ for \emph{all} large $i$ and $j$ is essential to obtain periodicity by Fact \ref{fact:vanishingdeltas}.  Note that even the vanishing of $\Delta_ d(i,j)$ for infinitely many pairs $(i,j)$ does not necessarily imply eventual periodicity of $\{b_i\}_{i\geq 0}$.

(2) The condition that $\{b_i\}_{i\geq 0}$ is eventually non-constant is equivalent to the non-vanishing of $\Delta_1(i,j)$ for infinitely many pairs $(i,j)$.

(3) If $\{b_i\}_{i\geq 0}$ is eventually strictly increasing, then $\Delta_1(i,j)\neq 0$ for all $i,j\gg 0$.
\end{remarks}

\begin{theorem}\label{thm:periodicity}
Let $M$ be a finitely generated module with an eventually linear minimal free resolution over a local ring $(R,\fm)$ satisfying $\fm^4=0$.  If $\Ext_R^i(M,R)$ vanishes for all $i\gg 0$, then the Betti sequence of $M$ is not eventually periodic of period two or three.
\end{theorem}

\begin{proof}
Assume that $\pd_R M=\infty$, and let the Betti sequence of $M$ be denoted $\{b_i\}_{i\geq 0}$.  By Fact \ref{fact:vanishingdeltas}, it suffices to show that neither $\Delta_2(i,j)$ nor $\Delta_3(i,j)$, as defined in \eqref{setup:deltas}, can vanish for all $i,j\gg 0$.  If the sequence $\{b_i\}_{i\geq 0}$ is eventually non-constant, the set
\[
I=\{(n,m)\in\NN^2\mid\Delta_1(n,m)\neq 0\}
\]
has infinite cardinality.

First suppose that $\Delta_2(i,j)=0$ for all $i,j\gg 0$, therefore implying that $\{b_i\}_{i\geq 0}$ eventually has period two.  By \eqref{setup:redmatrixeq}, one has
\[
f=\frac{\Delta_3(n,m)}{\Delta_1(n,m)}=\frac{b_n b_m-b_{n+3}b_{m+3}}{b_{n+1}b_{m+3}-b_n b_{m+2}}=\frac{b_n b_m-b_{n+1}b_{m+1}}{b_{n+1}b_{m+1}-b_n b_m}=-1
\]
for all $(n,m)\in I$, which is absurd.

Next suppose that $\Delta_3(i,j)=0$ for all $i,j\gg 0$, implying that $\{b_i\}_{i\geq 0}$ eventually has period three.  By \eqref{setup:redmatrixeq},
\[
f=-\left(\frac{\Delta_2(n,m)}{\Delta_1(n,m)}\right)e=\left(\frac{b_{n+2}b_{m+3}-b_n b_{m+1}}{b_{n+1}b_{m+3}-b_n b_{m+2}}\right)e=\left(\frac{b_{n+2}b_m-b_n b_{m+1}}{b_{n+1}b_m-b_n b_{m+2}}\right)e
\]
for all $(n,m)\in I$.  Notice that since $I$ has infinite cardinality, it is no loss of generality to assume that the set
\[
J=\{n\in\NN\mid(n,m)\in I\text{ for some }m\in\NN\}
\]
also has infinite cardinality.  Now, for any $n\in J$, if $(n,n)\in I$ then the above equation reduces to $f=-e$, which cannot be true.  So it must follow that $(n,n)\notin I$ for any $n\in J$.  This implies that
\[
\Delta_1(n,n)=b_{n+1}b_{n+3}-b_nb_{n+2}=b_{n}(b_{n+1}-b_{n+2})=0
\]
or $b_{n+1}=b_{n+2}$ for all $n\in J$.  Because the cardinality of $I$ must be infinite and we assumed that $\{b_i\}_{i\ge 0}$ has period three, it follows that $(n,n+1)=(n,n+2)\in I$.
\begin{align*}
(n,n+1)\in I\quad&\implies\quad f=\left(\frac{b_{n+2}b_{n+1}-b_n b_{(n+1)+1}}{b_{n+1}b_{n+1}-b_n b_{(n+1)+2}}\right)e=\left(\frac{b_{n+2}}{b_{n+1}+b_n}\right)e\\
(n,n+2)\in I\quad&\implies\quad f=\left(\frac{b_{n+2}b_{n+2}-b_n b_{(n+2)+1}}{b_{n+1}b_{n+2}-b_n b_{(n+2)+2}}\right)e=\left(\frac{b_{n+2}+b_n}{b_{n+1}}\right)e
\end{align*}
These relations hold simultaneously if and only if either $b_n=0$ or $b_n+b_{n+1}+b_{n+2}=0$ for all $n\in J$.  Since both of these are impossible, we have reached a contradiction.
\end{proof}

\subsection{Growth rates}

Among the mystery surrounding the Betti sequence of a finitely generated module over an arbitrary local ring is the following open question of Avramov.

\begin{question}\cite{Av10}*{4.3.7}
Does there exist a finitely generated module over a local ring whose Betti sequence grows subexponentially but superpolynomially?
\end{question}

It is known that the Betti sequence of a module over a local ring cannot grow superexponentially by work of Serre \cite{Ser56II}.

\begin{definition}\label{def:polynomial}
The Betti sequence $\{b_i\}_{i\geq 0}$ of a finitely generated $R$-module $M$ is said to have \emph{polynomial growth} if there exists $n\in\NN$ such that, for all $i\gg 0$,
\[
\alpha i^n-\lambda_i\leq b_i\leq\alpha i^n+\lambda_i
\]
for some $\alpha\in\RR^+$ and some sequence $\{\lambda_i\}_{i\geq 0}$ of real numbers satisfying $\lambda_i/i^n\to 0$.
\end{definition}

\begin{definition}\label{def:exponential}
Let $1<a\in\RR$.  The Betti sequence $\{b_i\}_{i\geq 0}$ of a finitely generated $R$-module $M$ is said to have \emph{exponential growth} (\emph{of base $a$}) if, for all $i\gg 0$,
\[
\beta a^i-\rho_i\leq b_i\leq\beta a^i+\rho_i
\]
for some $\beta\in\RR^+$ and some sequence $\{\rho_i\}_{i\geq 0}$ of real numbers satisfying $\rho_i/a^i\to 0$.
\end{definition}

\begin{remarks}\label{rmk:growthrates}
(1) The literature often refers to such growth rates in the language of complexity and curvature; cf.\ \cite{Av10}*{4.2}.  Whereas these quantities specify a smallest upper bound for the asymptotic behavior of certain Betti sequences, our definitions above provide a largest lower bound as well.

(2) Both of the above growth rates have been extensively studied.  It is well-known that over a complete intersection ring, every finitely generated module has a Betti sequence which grows polynomially \cite{Gul74}.  Furthermore, exponential growth of Betti numbers has been demonstrated in a variety of settings, including over Golod rings \cite{Sun98}, Cohen-Macaulay rings of small multiplicity \cites{GaPe90,Pe98}, and certain $\fm^3=0$ local rings \cite{Les85}.
\end{remarks}

%
%
%

Finally,  {we} define a special type of polynomial growth which will be of importance to our results in the next section.

\begin{definition}
The Betti sequence $\{b_i\}_{i\geq 0}$ of a finitely generated $R$-module $M$ is said to be \emph{exceptional} if
\[
b_{i+1}-b_{i}=b_{i+3}-b_{i+2}
\]
for all $i\gg 0$.
\end{definition}

\begin{remark}
Any Betti sequence which is either constant or exactly linear (that is, $b_{i+1}=b_i+\alpha$ for some $\alpha\in\NN$) is exceptional.  However, a Betti sequence $\{b_i\}_{i\geq 0}$ such that
\[
b_{i+1}-b_{i}=b_{i+3}-b_{i+2}\neq b_{i+2}-b_{i+1}
\]
for all $i\gg 0$ is also exceptional.  Though such growth of Betti numbers may seem pathological, it is known to occur over certain codimension two complete intersections by work of Avramov and Buchweitz \cite{AvBu00}.
\end{remark}

\section{The Hilbert Polynomial}\label{sec:hilbert}

\numberwithin{theorem}{section}
\numberwithin{equation}{theorem}

The ultimate goal of this section is to investigate necessary conditions which must be placed on the Hilbert polynomial of an $\fm^4=0$ local ring $R$ in order for the ring to admit a finitely generated module $M$ having an eventually linear resolution which is \emph{partially} complete --- in other words, satisfying $\Ext_R^i(M,R)=0=\Ext_R^i(M^*,R)$ for all $i\gg 0$.  However, there is much to be said about the Hilbert polynomial of such a ring in the more general setting.  That is, before considering the `partially complete' condition,  {we} first study the existence of $R$-modules $M$, with eventually linear resolutions, satisfying the vanishing of $\Ext_R^i(M,R)$ for $i\gg 0$.  The following section investigates an even more general scenario:  {we} don't require the vanishing of $\Ext_R^i(M,R)$ for any $i$.

\subsection{The general form}

\numberwithin{theorem}{subsection}
\numberwithin{equation}{theorem}

The recursion relation suggested by Proposition \ref{prop:linres}(1) gives one the ability to express the general form for the Hilbert polynomial of an $\fm^4=0$ local ring $R$ which admits a finitely generated module $M$ with an eventually linear minimal free resolution as
\begin{equation}\label{eq:general}
H_R(t)=1+et+ft^2+\left(\frac{b_{i+1}}{b_i}f-\frac{b_{i+2}}{b_i}e+\frac{b_{i+3}}{b_i}\right)t^3
\end{equation}
for any $i\gg 0$, where $\{b_i\}_{i\geq 0}$ denotes the Betti sequence of $M$.  If one furthermore assumes that the Betti sequence of $M$ has either polynomial or exponential growth, the following result is obtained.

\begin{lemma}\label{lem:polyexplinres}
Let $M$ be a finitely generated module with an eventually linear minimal free resolution over a local ring $(R,\fm)$ satisfying $\fm^4=0$.
\begin{enumerate}
\item If the Betti sequence of $M$ has polynomial growth, then
\[
H_R(t)=1+et+ft^2+(f-e+1)t^3.
\]
\item If the Betti sequence of $M$ has exponential growth of base $a$, then
\[
H_R(t)=1+et+ft^2+(af-a^2e+a^3)t^3.
\]
\end{enumerate}
\end{lemma}

\begin{proof}
Let the Betti sequence of $M$ be denoted by $\{b_i\}_{i\geq 0}$.  First  {we} prove (1).  By the hypothesis, there exists $n\in\NN$ such that, for all $i\gg 0$,
\[
\alpha i^n-\lambda_i\leq b_i\leq\alpha i^n+\lambda_i
\]
for some $\alpha\in\RR^+$ and some sequence $\{\lambda_i\}_{i\geq 0}$ satisfying $\lambda_i/i^n\to 0$.  Given these quantities, one obtains the following bound.
\begin{align*}
g&=\frac{b_{i+1}e-b_{i+2}f+b_{i+3}}{b_i}\\[0.25\baselineskip]
&\leq\frac{\left(\alpha(i+1)^n+\lambda_{i+1}\right)f-\left(\alpha(i+2)^n-\lambda_{i+2}\right)e+\left(\alpha(i+3)^n+\lambda_{i+3}\right)}{\alpha i^n-\lambda_i}
\end{align*}

Notice that the quantity on the right-hand side can be made arbitrarily close to $f-e+1$ as $i\to\infty$, and one can similarly show that $g$ is bounded from below by a quantity that asymptotically approaches $f-e+1$.  Hence $H_R(t)=1+et+ft^2+(f-e+1)t^3$, which is what was to be proved.

To show (2), let $1<a\in\RR$ be such that, for all $i\gg 0$,
\[
\beta a^i-\rho_i\leq b_i\leq\beta a^i+\rho_i
\]
for some $\beta\in\RR^+$ and some sequence $\{\rho_i\}_{i\geq 0}$ satisfying $\rho_i/a^i\to 0$.  As in the proof of (1),  {we} proceed to bound $g$.
\begin{align*}
g&=\frac{b_{i+1}e-b_{i+2}f+b_{i+3}}{b_i}\\[0.25\baselineskip]
&\leq\frac{\left(\beta a^{i+1}+\rho_{i+1}\right)f-\left(\beta a^{i+2}-\rho_{i+2}\right)e+\left(\beta a^{i+3}+\rho_{i+3}\right)}{\beta a^i-\rho_i}
\end{align*}

Therefore, $g$ is bounded from above by a quantity that can be made arbitrarily close to $af-a^2e+a^3$ as $i\to\infty$.  The same can be shown for a lower bound of $g$.  Thus, $H_R(t)=1+et+ft^2+(af-a^2e+a^3)t^3$, as claimed.
\end{proof}

 {We} illustrate the application of the characterizations provided by Lemma \ref{lem:polyexplinres} in the following example.

\begin{example}\label{ex:expnovanishext}
Let $R=k\llbracket w,x,y,z\rrbracket/(w^2,wx,x^2,y^2,z^2)$ and consider the $R$-module $M=R/(w,x)$.  One can use an inductive argument to show that the $n$th map in the minimal free resolution of $M$ over $R$ is represented by the block diagonal matrix
\[
\left[\begin{array}{ccccccc}w & x & & & & & \\ &  & w & x & & \\& & & & \ddots\\ & & & & & w & x\end{array}\right]_{n\times 2n}
\]
with respect to the standard bases of $R^n$ and $R^{2n}$, respectively.  Thus, $M$ has a linear minimal free resolution and its Betti sequence has exponential growth of base two.  One can now use Lemma \ref{lem:polyexplinres}(2) to recover the last coefficient of $H_R(t)$.
\begin{align*}
H_R(t)&=1+4t+5t^2+2t^3\\
&=1+4t+5t^2+(2\cdot 5-2^2\cdot 4+2^3)t^3
\end{align*}
\end{example}

Notice that the statement of Lemma \ref{lem:polyexplinres} does not assume anything about the vanishing of $\Ext_R^i(M,R)$ for $i\gg 0$.  In the next section, {we} shall investigate the additional restrictions on $H_R(t)$ which arise if one makes this assumption.

\subsection{(Eventual) vanishing of Ext}

The results in this section rely on the system in \eqref{setup:redmatrixeq} having full rank.  Since this system reduces to a single equation whenever the Betti sequence is eventually constant,  {we} shall restrict our attention to Betti sequences which are eventually non-constant.

The following proposition specifies conditions under which the Hilbert polynomial of a local ring $R$ can be expressed in terms of the embedding dimension of $R$ and just four Betti numbers of a suitable $R$-module.  This result provides a foundation for the results in the remainder of this manuscript.

\begin{proposition}\label{prop:genformext}
Let $M$ be a finitely generated module with an eventually linear minimal free resolution over a local ring $(R,\fm)$ satisfying $\fm^4=0$. Suppose that the Betti sequence of $M$ is not eventually constant and that $\Ext_R^i(M,R)=0$ for all $i>m\geq 0$.  If there exists $n\ge\max\{\ld_R(M),m\}$ such that $\Delta_1(n,n)\neq 0$, then $H_R(t)=1+et+ft^2+gt^3$, where
\begin{align*}
f&=\frac{(b_{n+2}b_{n+3}-b_nb_{n+1})e-(b_{n+3}^2-b_n^2)}{b_{n+1}b_{n+3}-b_nb_{n+2}}\\[0.25\baselineskip]
g&=\frac{(b_{n+2}^2-b_{n+1}^2)e-(b_{n+2}b_{n+3}-b_nb_{n+1})}{b_{n+1}b_{n+3}-b_nb_{n+2}}
\end{align*}
given the Betti sequence $\{b_i\}_{i\geq 0}$ of $M$.
\end{proposition}

\begin{proof}
It suffices to solve the system of equations in \eqref{setup:redmatrixeq}.  To this end,
\[
f=-\left(\frac{\Delta_2(n,n)}{\Delta_1(n,n)}\right)e+\frac{\Delta_3(n,n)}{\Delta_1(n,n)}=\frac{(b_{n+2}b_{n+3}-b_nb_{n+1})e-(b_{n+3}b_{n+3}-b_nb_n)}{b_{n+1}b_{n+3}-b_nb_{n+2}}
\]
and
\[
g=\left(\frac{b_{n+1}}{b_n}\right)f-\left(\frac{b_{n+2}}{b_n}\right)e+\frac{b_{n+3}}{b_n}.
\]
which imply the result upon simplification.
\end{proof}

\begin{example}
Let $R=k\llbracket w,x,y,z\rrbracket/(w^2,wx,x^2,y^2,z^2)$ be as in Example \ref{ex:expnovanishext}, but consider the $R$-module $M=R/(x)$.  Since $\Ext_R^i(M,R)$ does not eventually vanish, one would not expect Proposition \ref{prop:genformext} to recover the last two coefficients of $H_R(t)$.  Indeed,
\begin{align*}
f&=5\neq\frac{19}{4}=\frac{(b_{3}b_{4}-b_1b_{2})e-(b_{4}^2-b_1^2)}{b_{2}b_{4}-b_1b_{3}}\\[0.25\baselineskip]
g&=2\neq\frac{3}{2}=\frac{(b_{3}^2-b_{2}^2)e-(b_{3}b_{4}-b_1b_{2})}{b_{2}b_{4}-b_1b_{3}}.
\end{align*}
\end{example}

\begin{remark}
Indeed, it is possible to cook up a non-constant sequence $\{b_i\}_{i\geq 0}$ such that $\Delta_1(n,n)=b_nb_{n+2}-b_{n+1}b_{n+3}$ vanishes for all $n\geq 0$; in particular, such a sequence would be periodic of period four.  Although Proposition \ref{prop:genformext} could not be used explicitly in the presence of such a Betti sequence, one could derive a similar result by exploiting the non-vanishing of $\Delta_1(m,n)$ for some pair $(m,n)$.  Recall that the existence of such a pair $(m,n)$ is guaranteed by Remark \ref{rmk:delta}(2).  The following lemma demonstrates that $\Delta_1(n,n)$ vanishes for all $n\gg 0$ whenever the Betti sequence grows `fast enough.'
\end{remark}

\begin{lemma}\label{lem:delta1nonzero}
If the sequence $\{b_i\}_{i\geq 0}$ has either non-constant polynomial or exponential growth, then $\Delta_1(i,i)=b_ib_{i+2}-b_{i+1}b_{i+3}\neq 0$ for all $i\gg 0$.
\end{lemma}

\begin{proof}
Suppose that $\{b_i\}_{i\geq 0}$ has polynomial growth.  According to Definition \ref{def:polynomial}, there exists $n\in\NN$ such that, for all $i\gg 0$,
\[
\alpha i^n-\lambda_i\leq b_i\leq\alpha i^n+\lambda_i
\]
for some $\alpha\in\RR^+$ and some sequence $\{\lambda_i\}_{i\geq 0}$ satisfying $\lambda_i/i^n\to 0$.  This implies the following bound.
\[
\frac{b_i}{b_{i+1}}\leq\frac{\alpha i^n+\lambda_i}{\alpha(i+1)^n-\lambda_{i+1}}
\]
One can now see that $\ds\frac{b_i}{b_{i+1}}\to 0$, and therefore $\ds\frac{b_ib_{i+2}}{b_{i+1}b_{i+3}}\to 0$ as well.  Thus, $b_ib_{i+2}<b_{i+1}b_{i+3}$ for all $i\gg 0$.

A similar proof shows the same result whenever $\{b_i\}_{i\geq 0}$ has exponential growth.
\end{proof}

 {We} now consider the characterization in Proposition \ref{prop:genformext} given prescribed behavior in the Betti sequence.   {We} begin by showing that a symmetric Hilbert polynomial is necessary for non-exceptional polynomially growing Betti numbers.

\begin{theorem}\label{thm:polyext}
Let $M$ be a finitely generated module with an eventually linear minimal free resolution over a local ring $(R,\fm)$ satisfying $\fm^4=0$, and suppose that the Betti sequence of $M$ has non-exceptional polynomial growth.  If $\Ext_R^i(M,R)$ vanishes for all $i\gg 0$, then
\[
H_R(t)=1+et+et^2+t^3.
\]
\end{theorem}

\begin{proof}
By virtue of Lemma \ref{lem:polyexplinres}(1), the Hilbert polynomial of $R$ must take the form
\[
H_R(t)=1+et+ft^2+(f-e+1)t^3.
\]
 {We} will use this fact, along with the statement of Proposition \ref{prop:genformext}, to show that $f=e$, implying the result.

Let $\{b_i\}_{i\geq 0}$ denote the Betti sequence of $M$.  The general form of $H_R(t)$ given in \eqref{eq:general} implies that
\[
g=\left(\frac{b_{i+1}}{b_i}\right)f-\left(\frac{b_{i+2}}{b_i}\right)e+\frac{b_{i+3}}{b_i}=f-e+1
\]
for $i\gg 0$.  Solving for $f$ now yields
\[
f=\left(\frac{b_{i+2}-b_i}{b_{i+1}-b_i}\right)e-\frac{b_{i+3}-b_i}{b_{i+1}-b_i}
\]
for $i\gg 0$.  Since, by Lemma \ref{lem:delta1nonzero}, $\Delta_1(i,i)$ is non-vanishing for all $i\gg 0$, one can now equate the expression for $f$ above with that from Proposition \ref{prop:genformext} to obtain the equality
\[
\frac{(b_{i+2}b_{i+3}-b_ib_{i+1})e-(b_{i+3}^2-b_i^2)}{b_{i+1}b_{i+3}-b_ib_{i+2}}=\frac{(b_{i+2}-b_i)e-(b_{i+3}-b_i)}{b_{i+1}-b_i}
\]
for all $i\gg 0$.  This now implies the equation
\begin{align*}
(b_ib_{i+2}-b_{i+2}^2-b_{i+1}&b_{i+3})e-(b_ib_{i+2}-b_{i+2}b_{i+3}-b_{i+1}b_{i+3})\\
&=(b_ib_{i+1}-b_{i+1}^2-b_{i+2}b_{i+3})e-(b_i^2-b_ib_{i+1}-b_{i+3}^2)
\end{align*}
and therefore
\[
(b_{i+2}-b_{i+1})(b_{i+1}+b_{i+2}-b_i-b_{i+3})e=(b_{i+3}-b_i)(b_{i+1}+b_{i+2}-b_i-b_{i+3})
\]
for all $i\gg 0$. Now, since  {we} have assumed that $\{b_i\}_{i\geq 0}$ is not exceptional, it follows that that $b_{i+1}+b_{i+2}-b_i-b_{i+3}$ does not vanish infinitely often.  Thus, there exists some $j\gg 0$ such that
\[
e=\frac{b_{j+3}-b_j}{b_{j+2}-b_{j+1}}.
\]
Substituting this value of $e$ into the the expression for $f$ which is obtained by solving the second equation of \eqref{setup:redmatrixeq}, one has the following upon simplification.
\begin{align*}
f&=\left(\frac{b_{j+2}-b_j}{b_{j+1}-b_j}\right)\left(\frac{b_{j+3}-b_j}{b_{j+2}-b_{j+1}}\right)-\frac{b_{j+3}-b_1}{b_{j+1}-b_j}\\[0.1in]
&=\frac{b_{j+3}-b_j}{b_{j+2}-b_{j+1}}\\[0.05in]
&=e
\end{align*}
This, of course, implies the result.
\end{proof}

The following two examples illustrate Theorem \ref{thm:polyext}.

\begin{example}\label{ex:notexc}
Let $R=k\llbracket x,y,z\rrbracket/(x^2,y^2,z^2)$.  Then $k$ is a totally reflexive $R$-module; in particular, $\Ext_R^i(k,R)=0$ for $i>0$.

By virtue of the fact that $R$ is a complete intersection ring, the Betti sequence of $k$ has polynomial growth; cf. Remark \ref{rmk:growthrates}(2).  Even better than this, as $R$ is Koszul, one can explicitly write down the Poincar\'e series of $k$.
\[
P_k^R(t)=\frac{1}{H_R(-t)}=\frac{1}{(1-t)^3}=\sum_{i\geq 0}{{i+2}\choose 2} t^i=\textstyle{\frac{1}{2}}\displaystyle\sum_{i\geq 0}(i^2+3i+2)\,t^i
\]
Since the Betti sequence of $k$ has quadratic growth, it is not exceptional. Further, the above resolution must be linear since $R$ is Koszul.  Also note that the Hilbert polynomial of $R$, given by $H_R(t)=1+3t+3t^2+t^3$, is symmetric.
\end{example}

%

\begin{example}
Let $R=k\llbracket w,x,y,z\rrbracket/(w^2,wx,x^2,y^2,z^2)$, which has an embedded deformation given by $k\llbracket w,x,y,z\rrbracket/(w^2,x^2,wx)=S\onto S/(y^2,z^2)\cong R$.  If one defines $M=R/(y,z)$, then one can check that $\Ext_R^i(M,R)$ vanishes for $i>0$.  Furthermore, the minimal free resolution of $M$ over $R$ is given by the following sequence.
\[
\cdots\to
R^4\xra{\left[\footnotesize{\begin{array}[c]{@{\hspace{0.25em}}c@{\hspace{1em}}c@{\hspace{0.5em}}r@{\hspace{0.5em}}r@{\hspace{0.25em}}}y&0&z&0\\0&z&0&y\\0&0&-y&-z\end{array}}\right]}
R^3\xra{\left[\footnotesize{\begin{array}[c]{@{\hspace{0.25em}}c@{\hspace{1em}}c@{\hspace{0.5em}}r@{\hspace{0.25em}}}y&0&z\\0&z&-y\end{array}}\right]}
R^2\xra{\left[\footnotesize{\begin{array}[c]{@{\hspace{0.25em}}c@{\hspace{1em}}c@{\hspace{0.25em}}}y&z\end{array}}\right]}
R\to M\to 0
\]
At this point, the Poincar\'e series of $M$ might be fairly obvious.  However, to be thorough set $P=k\llbracket w,x\rrbracket/(w^2,wx,x^2)$ and $Q=k\llbracket y,z\rrbracket/(y^2,z^2)$, and notice that $P\otimes_k Q\cong R$ and $P\cong M$ as $k$-algebras.  Therefore, let $\FF\onto k$ be a minimal $Q$-free resolution of $k\cong Q/(y,z)$.  The ranks of the free modules in $\FF$ are well-understood since $k$ is the residue field of $Q$;  {we} exhibit them in the following Poincar\'e series.
\begin{equation}\label{eq:poincare}
P_Q(t)=\frac{1}{H_{\gr_{\fm}(Q)}(-t)}=\frac{1}{(1-t)^2}=\sum_{i\geq 0}(i+1)\,t^i
\end{equation}
Since $\Tor^k_i(P,Q)=0$ for all $i>0$, $\FF\otimes_k P$ is a minimal free resolution of $M$ over $R$.  One can therefore conclude that the Poincar\'e series of $M$ over $R$ is the same as the one given in \eqref{eq:poincare}.  In particular, the Betti sequence of $M$ is exceptional.  Furthermore, recall that the Hilbert polynomial of $R$ is not symmetric; in fact, one has that $H_R(t)=1+4t+5t^2+2t^3$.
\end{example}

 {We} now consider the characterization of the Hilbert polynomial of an $\fm^4=0$ local ring whenever it admits modules with exponentially growing Betti numbers.

\begin{theorem}\label{thm:expext}
Let $M$ be a finitely generated module with an eventually linear minimal free resolution over a local ring $(R,\fm)$ with $\fm^4=0$, and suppose that the Betti sequence of $M$ has exponential growth of base $a$.  If $\Ext_R^i(M,R)$ vanishes for all $i\gg 0$, then $H_R(t)=1+et+ft^2+gt^3$, where
\begin{align*}
f&=\left(a+\frac{1}{a}\right)e-\left(a^2+1+\frac{1}{a^2}\right)\\[0.5\baselineskip]
g&=e-\left(a+\frac{1}{a}\right).
\end{align*}
\end{theorem}

\begin{proof}
Let $\{b_i\}_{i\geq 0}$ denote the Betti sequence of $M$.  By assumption, for all $i\gg 0$ one has
\begin{equation}\label{eq:expbound}
\beta a^i-\rho_i\leq b_i\leq\beta a^i+\rho_i
\end{equation}
for some $\beta\in\RR^+$ and some sequence $\{\rho_i\}_{i\geq 0}$ of real numbers satisfying $\rho_i/a^i\to 0$.

Since Lemma \ref{lem:delta1nonzero} guarantees that $\Delta_1(i,i)\neq 0$ for all $i\gg 0$,  {we} proceed by bounding the expressions for $f$ and $g$ found in Proposition \ref{prop:genformext} using the bound in \eqref{eq:expbound}.  In the interest of space,  {we} omit the tedious details which are analogous to those found in the proof of Lemma \ref{lem:polyexplinres}(2).  Indeed, one obtains the following expressions
\begin{align*}
f&=\frac{(a^5-a)e-(a^6-1)}{a^4-a^2}\\[0.5\baselineskip]
g&=\frac{(a^4-a^2)e-(a^5-a)}{a^4-a^2}
\end{align*}
which simplify to yield the result.
\end{proof}

An immediate corollary to the previous result is apparent if  {we} consider the fact that $f$ and $g$ must be positive integers.

\begin{corollary}
Let $M$ be a finitely generated module with an eventually linear minimal free resolution over a local ring $(R,\fm)$ with $\fm^4=0$, and suppose that $\Ext_R^i(M,R)$ vanishes for all $i\gg 0$.  If the Betti sequence of $M$ has exponential growth of base $a$, then
\[
a=r+s\sqrt{\alpha}
\]
for some $r,s\in\QQ$ and some $\alpha\in\ZZ^+$ satisfying $r^2-\alpha s^2=1$.
\end{corollary}

\begin{proof}
By Theorem \ref{thm:expext}: (1) $a$ must satisfy some quadratic equation with integer coefficients --- thus, one can write $a=r+s\sqrt{\alpha}$ for some $r,s\in\QQ$ and $\alpha\in\NN$; (2) the sum of $a$ and its reciprocal must be an integer.

If $a\in\QQ$, one can write $a=\displaystyle{\frac{p}{q}}$, where $p,q\in\ZZ^+$ are relatively prime.  Then
\[
a+\frac{1}{a}=\frac{p}{q}+\frac{q}{p}=\frac{p^2+q^2}{pq}\in\ZZ
\]
which implies that $p^2-npq+q^2=0$ for some $n\in\ZZ$.  Solving for $p$ yields
\[
p=\frac{nq\pm\sqrt{n^2q^2-4q^2}}{2}=\frac{nq\pm q\sqrt{n^2-4}}{2}.
\]
In order for this quantity to be an integer, it must be true that $n^2-4$ is a perfect square.  Since there is no Pythagorean triple of the form $(2,m,n)$, it follows that $n=2$, which corresponds to the case that $p=q=1$ --- a contradiction.

Since $a$ is irrational, one can assume that $s,\alpha\neq 0$.  Furthermore, it is an easy exercise to check that
\[
a+\frac{1}{a}=\frac{r(r^2-\alpha s^2+1)+s(r^2-\alpha s^2-1)\sqrt{\alpha}}{r^2-\alpha s^2}
\]
which must be an integer, whence it follows that $r^2-\alpha s^2=1$.
\end{proof}

In light of Lemma \ref{lem:polyexplinres}, if one wishes to find an $\fm^4=0$ local ring which admits linear resolutions and has a Hilbert polynomial which is \emph{not} balanced, it would be natural to expect the ring to only admit exponentially growing Betti sequences.  In fact, Theorem \ref{thm:expext} does not even guarantee that the Hilbert polynomial of such a ring is balanced in the case that the module satisfies the vanishing of Ext condition.  The next example illustrates this scenario.

\begin{example}\label{ex:notbal}
Define local rings $S=k\llbracket x,y,z\rrbracket/(x^2-y^2,x^2-z^2,xy,xz,yz)$ and $Q=k\llbracket u,v\rrbracket/(u^2,uv,v^2)$, with maximal ideals $\fm_S=(x,y,z)$ and $\fm_Q=(u,v)$, respectively.  According to \cite[Construction 3.1]{Be12}, the local ring
\[
R:=S\otimes_k Q\cong k\llbracket u,v,w,x,y,z\rrbracket/(u^2,uv,v^2,x^2-y^2,x^2-z^2,xy,xz,yz)
\]
admits non-trivial totally reflexive modules; in particular, $M=R/(x,y,z)$ is one such module.

One would assume that the Betti sequence of $M$ over $R$ would coincide with that of $k\cong S/\fm_S$ over $S$.  To check this, first note that $M\cong Q$ as $k$-algebras.  Further, let $\FF\onto k$ be a minimal $S$-free resolution.  Since $\Tor^k_i(S,Q)$ vanishes for $i>0$, it follows that $\FF\otimes_k Q$ is a minimal free resolution of $M$ over $R$.  As $S$ is a Gorenstein ring satisfying $\fm_S^3=0$, the Betti sequence of the residue field $k$ over $S$ has exponential growth.  Therefore, the same must be true of the Betti sequence of $M$ over $R$.

Furthermore notice that one can easily check that the Hilbert polynomial
\[
H_R(t)=1+5t+7t^2+2t^3
\]
of $R$ is clearly not balanced.  Also, by using the statement of Theorem \ref{thm:expext}, one can recover the base $a$ of the exponential growth of the Betti sequence of $M$.  Indeed,
\begin{align*}
a&=\frac{e-g\pm\sqrt{((g-e)^2-4)}}{2}\\
&=\frac{3\pm\sqrt{5}}{2}
\end{align*}
which implies that $a=\frac{3}{2}+\frac{1}{2}\sqrt{5}$.
\end{example}

\begin{remark}\label{rmk:torindep}
Indeed, one can generalize the previous example.  To this end, let $S=k\llbracket x_1,\ldots,x_n\rrbracket/I$ and $Q=k\llbracket y_1,\ldots,y_m\rrbracket/J$ where $I$ is generated over $k\llbracket x_1,\ldots,x_n\rrbracket$ by $x_1^2-x_j^2$ and $x_ix_j$ for $0\leq i<j\leq n$, and where $J$ is generated over $k\llbracket y_1,\ldots,y_m\rrbracket$ by $y_iy_j$ for $1\leq i\leq j\leq n$.  It is clear to see that $S$ is a Gorenstein local ring with Hilbert polynomial $H_S(t)=1+nt+t^3$, and that $Q$ is a Cohen-Macaulay local ring with Hilbert polynomial $H_Q(t)=1+mt$.  Furthermore, since $S$ and $Q$ are Tor-independent $k$-algebras, the Hilbert polynomial of $R:=S\otimes_k Q$ is given by
\begin{align*}
H_R(t)&=H_S(t)\cdot H_Q(t)\\
&=(1+nt+t^2)(1+mt)\\
&=1+(n+m)t+(1+nm)t^2+mt^3
\end{align*}
which is only balanced if $m=1$ or if $n=2$.
\end{remark}

 {We} now turn our attention to necessary conditions for the existence of $R$-modules $M$ which satisfy both $\Ext_R^i(M,R)=0$ and $\Ext_R^i(M^*,R)=0$ for all $i\gg 0$.

\subsection{Asymmetric partially complete resolutions}\label{sec:asymmetric}

Our ultimate goal for this section is to investigate necessary conditions for an $\fm^4=0$ local ring to admit certain asymmetric (eventually) linear resolutions which are partially complete.  However, our actual results are even more general than this:  {we} only require that $R$ admit two modules of differing growth in their Betti sequences.

The idea behind the results in this section is that if a local ring admits modules satisfying the hypotheses of both Theorem \ref{thm:polyext} and Theorem \ref{thm:expext}, then its Hilbert polynomial must take on both of the respective forms.  It is straightforward to see that a study of linear vs.\ linear asymmetric resolutions will not reveal any additional information about the ring's Hilbert polynomial.  Therefore,  {we} restrict our investigation to the remaining two cases.    {We} begin by considering the sort of asymmetric growth of Betti numbers which is apparent in Example \ref{ex:joseg}: polynomial vs.\ exponential growth.

\begin{theorem}\label{thm:polyexp}
Let $M$ and $N$ be finitely generated modules, each with an eventually linear minimal free resolution over a local ring $(R,\fm)$ satisfying $\fm^4=0$, and suppose that $\Ext_R^i(N,R)$ vanishes for $i\gg0$. If the Betti sequence of $M$ has polynomial growth, whereas the Betti sequence of $N$ has exponential growth, then
\[
H_R(t)=1+et+et^2+t^3.
\]
\end{theorem}

\begin{proof}
By Lemma \ref{lem:polyexplinres}(1),  the Hilbert polynomial of $R$ must be balanced; that is, $H_R(t)=1+et+ft^2+(f-e+1)t^3$.  Furthermore, given the growth of the Betti sequence of $N$ and the fact that $\Ext_R^i(N,R)=0$ for $i\gg 0$, one can use the characterization of $f$ in Theorem \ref{thm:expext} to obtain
\begin{align*}
g&=f-e+1\\[0.5\baselineskip]
&=\left(a+\frac{1}{a}\right)e-\left(a^2+1+\frac{1}{a^2}\right)-e+1\\[0.5\baselineskip]
&=\left(a-1+\frac{1}{a}\right)e-\left(a^2+\frac{1}{a^2}\right).
\end{align*}
However, by Theorem \ref{thm:expext} one has $g=e-\left(a+\displaystyle{\frac{1}{a}}\right)$.  Equating these expressions for $g$ yields
\[
e=a+1+\frac{1}{a}.
\]
It is straightforward to check that the Hilbert polynomial of an $\fm^4=0$ local ring with this embedding dimension must be symmetric.
\end{proof}

\begin{remark}
In light of Theorem \ref{thm:polyexp}, it is impossible for the ring illustrated in Example \ref{ex:notbal} to admit Koszul modules with polynomially growing Betti sequences.  In particular, that this implies the ring does not have an exact pair of zero divisors.
\end{remark}


Our final result essentially states that asymmetric complete resolutions with exponential vs.\ exponential growth cannot occur.

\begin{theorem}\label{thm:expexp}
Let $M$ and $N$ be finitely generated modules, each with an eventually linear minimal free resolution over a local ring $(R,\fm)$ satisfying $\fm^4=0$, and  suppose that $\Ext_R^i(M,R)=0=\Ext_R^i(N,R)$ for all $i\gg0$.  If the Betti sequences of $M$ and $N$ have exponential growth of bases $a$ and $b$, respectively, then $a=b$.
\end{theorem}

\begin{proof}
Suppose the contrary.  By Theorem \ref{thm:expext} one has
\[
g=e-\left(a+\frac{1}{a}\right)=e-\left(b+\frac{1}{b}\right)
\]
which simplifies to yield $ab=1$.  Since both $a$ and $b$ must be larger than one,  {we} have reached a contradiction.
\end{proof}

%

\bibliographystyle{amsxport}
\bibliography{mybib}

\end{document}